\documentclass[a4paper,12pt,intlimits,oneside]{amsart}

\usepackage{enumerate}
\usepackage{amsfonts,amsmath}

\usepackage{latexsym,amssymb}

\usepackage{mathrsfs}

\newcommand{\comment}[1]{}

\newcommand{\R}{{\mathbb R}}

\begin{document}

\title[On weak$^*$-convergence in $h^1$]{A note on weak$^*$-convergence in $h^1(\R^d)$}         

\author{Ha Duy HUNG}    
\address{High School for Gifted Students,
	Hanoi National University of Education, 136 Xuan Thuy, Hanoi, Vietnam} 
\email{{\tt hunghaduy@gmail.com}}
 \author{Duong Quoc Huy} 
  \address{Department of Natural Science and Technology, Tay Nguyen University, Daklak, Vietnam.}
 \email{duongquochuy@ttn.edu.vn}
  \author{Luong Dang Ky $^*$}
 \address{Department of Mathematics,
 	University of Quy Nhon,
 	170 An Duong Vuong,
 	Quy Nhon, Binh Dinh, Vietnam} 
 \email{{\tt dangky@math.cnrs.fr}}
 \date{}
 \thanks{The paper was completed when the third author was visiting
 	to Vietnam Institute for Advanced Study in Mathematics (VIASM). He would like
 	to thank the VIASM for financial support and hospitality.\\ 	
 	$^{*}$Corresponding author: Luong Dang Ky}

\begin{abstract}
In this paper, we give a very simple proof of the main result of Dafni (Canad Math Bull 45:46--59, 2002) concerning with weak$^*$-convergence in the local Hardy space $h^1(\R^d)$.
\end{abstract}

\maketitle
\newtheorem{theorem}{Theorem}[section]
\newtheorem{lemma}{Lemma}[section]
\newtheorem{proposition}{Proposition}[section]
\newtheorem{remark}{Remark}[section]
\newtheorem{corollary}{Corollary}[section]
\newtheorem{definition}{Definition}[section]
\newtheorem{example}{Example}[section]
\numberwithin{equation}{section}
\newtheorem{Theorem}{Theorem}[section]
\newtheorem{Lemma}{Lemma}[section]
\newtheorem{Proposition}{Proposition}[section]
\newtheorem{Remark}{Remark}[section]
\newtheorem{Corollary}{Corollary}[section]
\newtheorem{Definition}{Definition}[section]
\newtheorem{Example}{Example}[section]
\newtheorem*{theoremjj}{Theorem J-J}
\newtheorem*{theorema}{Theorem A}
\newtheorem*{theoremb}{Theorem B}
\newtheorem*{theoremc}{Theorem C}

\section{Introduction}

A famous and classical result of Fefferman \cite{Fef} states that the John-Nirenberg space $BMO(\mathbb R^d)$ is the dual of the Hardy space $H^1(\mathbb R^d)$. It is also well-known that $H^1(\mathbb R^d)$ is one of the few examples of separable, nonreflexive Banach space which is a dual space. In fact,  let $C_c(\R^d)$ be the space of all continuous functions with compact support and denote by $VMO(\mathbb R^d)$ the closure of $C_c(\R^d)$ in $BMO(\mathbb R^d)$, Coifman and Weiss showed in \cite{CW} that $H^1(\mathbb R^d)$ is the dual space of $VMO(\mathbb R^d)$, which gives to $H^1(\mathbb R^d)$ a richer structure than $L^1(\mathbb R^d)$. For example, the classical Riesz transforms $\nabla (-\Delta)^{-1/2}$ are not bounded on $L^1(\mathbb R^d)$, but are bounded on  $H^1(\mathbb R^d)$. In addition, the weak$^*$-convergence is true in $H^1(\mathbb R^d)$ (see \cite{JJ}), which is useful  in the application of Hardy spaces to compensated compactness (see \cite{CLMS}) and in studying the endpoint estimates for commutators of singular integral operators (see \cite{Ky1, Ky2, Ky3}). Recently, Dafni showed in \cite{Da} that the local Hardy space $h^1(\R^d)$ of Goldberg \cite{Go} is in fact the dual space of $vmo(\R^d)$ the closure of $C_c(\R^d)$ in $bmo(\mathbb R^d)$. Moreover,  the weak$^*$-convergence is true in $h^1(\mathbb R^d)$. More precisely, in \cite{Da}, the author proved:

\begin{theorem}\label{Dafni 1}
The space $h^1(\R^d)$ is the dual of the space $vmo(\R^d)$.
\end{theorem}

\begin{theorem}\label{Dafni 2}
	Suppose that $\{f_n\}_{n=1}^\infty$ is a bounded sequence in $h^1(\mathbb R^d)$, and that $\lim_{n\to\infty} f_n(x) = f(x)$ for almost every $x\in\mathbb R^d$. Then, $f\in h^1(\mathbb R^d)$ and $\{f_n\}_{n= 1}^\infty$ weak$^*$-converges to $f$, that is, for every $\phi\in vmo(\mathbb R^d)$, we have
	$$\lim_{n\to\infty} \int_{\mathbb R^d} f_n(x) \phi(x)dx = \int_{\mathbb R^d} f(x) \phi(x) dx.$$
\end{theorem}

The aim of the present paper is to give very simple proofs of the two above theorems. It should be pointed out that our method is different from that of Dafni and it can be generalized to the setting of spaces of homogeneous type (see \cite{Ky4}). 

To this end, we first recall some definitions of the function spaces. As usual, $\mathcal S(\mathbb R^d)$ denotes the Schwartz class of test functions on $\mathbb R^d$. The subspace $\mathcal A$ of $\mathcal S(\mathbb R^d)$ is then defined by
$$\mathcal A=\Big\{\phi\in \mathcal S(\mathbb R^d): |\phi(x)|+ |\nabla\phi(x)|\leq (1+ |x|^2)^{-(d+1)}\Big\},$$
where $\nabla= (\partial/\partial x_1,..., \partial/\partial x_d)$ denotes the gradient. We define
$$\mathfrak M f(x):= \sup\limits_{\phi\in\mathcal A}\sup\limits_{|y-x|<t}|f*\phi_t(y)|\quad\mbox{and}\quad  \mathfrak mf(x):= \sup\limits_{\phi\in\mathcal A}\sup\limits_{|y-x|<t<1}|f*\phi_t(y)|,$$
where $\phi_t(\cdot)= t^{-d}\phi(t^{-1}\cdot)$. The space $H^1(\mathbb R^d)$ is the space of all integrable functions $f$ such that $\mathfrak M f\in L^1(\mathbb R^d)$ equipped with the norm $\|f\|_{H^1}= \|\mathfrak M f\|_{L^1}$. The space $h^1(\mathbb R^d)$ denotes the space of all integrable functions $f$ such that $\mathfrak m f\in L^1(\mathbb R^d)$ equipped with the norm $\|f\|_{h^1}= \|\mathfrak m f\|_{L^1}$.

We remark that the local real Hardy space $h^1(\mathbb R^d)$, first introduced  by Goldberg \cite{Go}, is larger than $H^1(\mathbb R^d)$ and allows  more flexibility, since global cancellation conditions are not necessary. For example, the Schwartz class $\mathcal S(\R^d)$ is contained in $h^1(\mathbb R^d)$ but not in $H^1(\mathbb R^d)$, and multiplication by cutoff  functions preserves  $h^1(\mathbb R^d)$ but not $H^1(\mathbb R^d)$. Thus it makes $h^1(\mathbb R^d)$ more suitable for working in domains and on manifolds. 

It is well-known (see \cite{Fef}) that  the dual space of $H^1(\mathbb R^d)$ is $BMO(\mathbb R^d)$ the space of all  locally integrable functions $f$ with
$$\|f\|_{BMO}:=\sup\limits_{B}\frac{1}{|B|}\int_B \Big|f(x)-\frac{1}{|B|}\int_B f(y) dy\Big|dx<\infty,$$
where the supremum is taken over all balls $B\subset \R^d$. It was also shown in \cite{Go} that the dual space of $h^1(\mathbb R^d)$  can be identified with the space $bmo(\mathbb R^d)$, consisting of locally integrable functions $f$ with
$$\|f\|_{bmo}:= \sup\limits_{|B|\leq 1}\frac{1}{|B|}\int_B \Big|f(x)-\frac{1}{|B|}\int_B f(y) dy\Big|dx+ \sup\limits_{|B|\geq 1}\frac{1}{|B|}\int_B |f(x)|dx<\infty,$$
where the supremums are taken over all balls $B\subset \R^d$.

It is clear that, for any $f\in H^1(\R^d)$ and $g\in bmo(\R^d)$, 
$$\|f\|_{h^1} \leq \|f\|_{H^1}\quad\mbox{and}\quad \|g\|_{BMO} \leq \|g\|_{bmo}.$$

Recall that the space $VMO(\mathbb R^d)$ (resp., $vmo(\mathbb R^d)$) is the closure of $C_c(\mathbb R^d)$ in $(BMO(\mathbb R^d),\|\cdot\|_{BMO})$ (resp., $(bmo(\mathbb R^d),\|\cdot\|_{bmo})$). The following theorem is due to Coifman and Weiss \cite{CW}.

\begin{theorem}\label{Coifman-Weiss}
	The space $H^1(\R^d)$ is the dual of the space $VMO(\R^d)$.
\end{theorem}

Throughout the whole paper, $C$ denotes a positive geometric constant which is independent of the main parameters, but may change from line to line.


\section{Proof of Theorems \ref{Dafni 1} and \ref{Dafni 2}}

In this section, we fix $\varphi\in C_c(\R^d)$ with supp $\varphi\subset B(0,1)$ and $\int_{\R^d} \varphi(x) dx=1$. Let $\psi:= \varphi*\varphi$. The following lemma is due to Goldberg \cite{Go}.

\begin{lemma}\label{Golberg}
There exists a positive constant $C=C(d,\varphi)$ such that

{\rm i)} for any $f\in L^1(\R^d)$,
$$\|\varphi*f\|_{h^1}\leq C \|f\|_{L^1};$$

{\rm ii)} for any $g\in h^1(\R^d)$,
$$\|g- \psi*g\|_{H^1}\leq C \|g\|_{h^1}.$$
\end{lemma}

As a consequence of Lemma \ref{Golberg}(ii), for any $\phi\in C_c(\R^d)$, 
\begin{equation}\label{from big bmo to small bmo}
	\|\phi - \overline{\psi}*\phi\|_{bmo} \leq C \|\phi\|_{BMO},
\end{equation}
here and hereafter, $\overline{\psi}(x):= \psi(-x)$ for all $x\in \R^d$.

\begin{proof}[Proof of Theorem \ref{Dafni 1}]
Since $vmo(\mathbb R^d)$ is a subspace of  $bmo(\mathbb R^d)$, which is the dual space of $h^1(\mathbb R^d)$, every function $f$ in $h^1(\mathbb R^d)$ determines a bounded linear functional on $vmo(\mathbb R^d)$ of norm bounded by $\|f\|_{h^1}$.

Conversely, given a bounded linear functional $L$ on $vmo(\mathbb R^d)$. Then, 
$$|L(\phi)|\leq \|L\| \|\phi\|_{vmo}\leq \|L\| \|\phi\|_{L^\infty}$$
for all $\phi\in C_c(\mathbb R^d)$. This implies (see \cite{Ro}) that there exists a finite signed Radon measure $\mu$ on $\R^d$ such that, for any $\phi\in C_c(\mathbb R^d)$, 
$$L(\phi)= \int_{\R^d} \phi(x) d\mu(x),$$
moreover,  the total variation of $\mu$, $|\mu|(\R^d)$,  is bounded by $\|L\|$. Therefore, 
\begin{equation}\label{Dafni 1, 1}
\|\psi*\mu\|_{h^1} = \|\varphi*(\varphi*\mu)\|_{h^1} \leq C \|\varphi*\mu\|_{L^1}\leq C |\mu|(\R^d)\leq C \|L\|
\end{equation}
by Lemma \ref{Golberg}. On the other hand, by (\ref{from big bmo to small bmo}), we have
\begin{eqnarray*}
|(L-\psi*L)(\phi)|=|L(\phi - \overline{\psi}*\phi)|&\leq& \|L\| \|\phi - \overline{\psi}*\phi\|_{vmo}\\
&\leq& C \|L\| \|\phi\|_{BMO}
\end{eqnarray*}
for all $\phi\in C_c(\R^d)$. Consequently, by Theorem \ref{Coifman-Weiss}, there exists a function $h$ belongs $H^1(\R^d)$ such that $\|h\|_{H^1}\leq C \|L\|$ and
$$(L-\psi*L)(\phi)= \int_{\R^d} h(x) \phi(x) dx$$
for all  $\phi\in C_c(\R^d)$. This, together with (\ref{Dafni 1, 1}), allows us to conclude that
$$L(\phi)= \int_{\R^d} f(x) \phi(x) dx$$
for all $\phi\in C_c(\R^d)$, where $f:= h+ \psi*\mu\in h^1(\R^d)$ satisfying $\|f\|_{h^1}\leq \|h\|_{H^1} + \|\psi*\mu\|_{h^1}\leq C \|L\|$. The proof of Theorem \ref{Dafni 1} is thus completed.
\end{proof}

\begin{proof}[Proof of Theorem \ref{Dafni 2}]
Let $\{f_{n_k}\}_{k=1}^\infty$ be an arbitrary subsequence of $\{f_n\}_{n=1}^\infty$.	As $\{f_{n_k}\}_{k=1}^\infty$ is a bounded sequence in $h^1(\R^d)$, by Theorem \ref{Dafni 1} and the Banach-Alaoglu theorem, there exists a subsequence $\{f_{n_{k_j}}\}_{j=1}^\infty$ of $\{f_{n_k}\}_{k=1}^\infty$ such that $\{f_{n_{k_j}}\}_{j=1}^\infty$ weak$^*$-converges to $g$ for some $g\in h^1(\R^d)$. Therefore, for any  $x\in \R^d$,
$$\lim_{j\to\infty} \int_{\R^d} f_{n_{k_j}}(y) \psi(x-y) dy = \int_{\R^d} g(y) \psi(x-y) dy.$$
This implies that $\lim_{j\to\infty}[f_{n_{k_j}}(x) - (f_{n_{k_j}}*\psi)(x)]= f(x) - (g*\psi)(x)$ for almost every $x\in\mathbb R^d$. Hence, by Lemma \ref{Golberg}(ii) and the Jones-Journ\'e's theorem (see \cite{JJ}),  
$$\|f- g*\psi\|_{H^1}\leq \sup_{j\geq 1}\|f_{n_{k_j}} - f_{n_{k_j}}*\psi\|_{H^1}\leq C \sup_{j\geq 1} \|f_{n_{k_j}}\|_{h^1}<\infty,$$
moreover,
$$\lim_{j\to\infty} \int_{\mathbb R^d} [f_{n_{k_j}}(x) - (f_{n_{k_j}}*\psi)(x)]\phi(x)dx =  \int_{\mathbb R^d}  [f(x) - (g*\psi)(x)]\phi(x) dx$$
for all $\phi\in C_c(\mathbb R^d)$. As a consequence, we obtain that
\begin{eqnarray*}
\|f\|_{h^1} \leq \|f- g*\psi\|_{h^1} + \|g*\psi\|_{h^1}&\leq& \|f- g*\psi\|_{H^1} + C\|g\|_{h^1}\\
&\leq& C \sup_{j\geq 1} \|f_{n_{k_j}}\|_{h^1}<\infty,
\end{eqnarray*}
moreover,
\begin{eqnarray*}
	&&\lim_{j\to\infty} \int_{\mathbb R^d} f_{n_{k_j}}(x)\phi(x)dx \\
	&=& \lim_{j\to\infty} \int_{\mathbb R^d} [f_{n_{k_j}}(x) - (f_{n_{k_j}}*\psi)(x)]\phi(x)dx + \lim_{j\to\infty} \int_{\mathbb R^d} f_{n_{k_j}}(x) (\overline\psi*\phi)(x)dx\\
	&=& \int_{\mathbb R^d}  [f(x) - (g*\psi)(x)]\phi(x) dx + \int_{\mathbb R^d} g(x) (\overline\psi*\phi)(x)dx\\
	&=& \int_{\mathbb R^d} f(x)\phi(x)dx
\end{eqnarray*}
since $\{f_{n_{k_j}}\}_{j=1}^\infty$ weak$^*$-converges to $g$ in $h^1(\mathbb R^d)$. This, by $\{f_{n_k}\}_{k=1}^\infty$ be an arbitrary subsequence of $\{f_n\}_{n=1}^\infty$,  allows us to complete the proof of Theorem \ref{Dafni 2}.
\end{proof}



%

\end{document}